\documentclass[aop,preprint]{imsart}

\RequirePackage[OT1]{fontenc}
\RequirePackage{amsthm,amsmath,amsfonts}
\RequirePackage[numbers]{natbib}
\RequirePackage[colorlinks,citecolor=blue,urlcolor=blue]{hyperref}
\usepackage[mathscr]{eucal}
\usepackage{amsfonts}
\usepackage{amsmath,amsxtra,latexsym,amsthm, amssymb, amscd}
\usepackage{graphicx}
\usepackage{color}

\startlocaldefs
\numberwithin{equation}{section}
\theoremstyle{plain}
\newtheorem{theorem}{Theorem}[section]

\theoremstyle{definition}

\endlocaldefs
\allowdisplaybreaks

\begin{document}

\begin{frontmatter}
\title{An Ordinary Differential Equation Model for Fish Schooling}
\runtitle{Fish Schooling}

\begin{aug}
\author{Takeshi Uchitane$^{\dagger}$, Ta Viet Ton$^{\ddagger}$, Atsushi Yagi$^{\ddagger}$\thanksref{t2}   }
\runauthor{Takeshi Uchitane, Ta Viet Ton and Atsushi Yagi}
\thankstext{t2}{This work is supported by Grant-in-Aid for Scientific Research (No. 20340035) of the Japan Society for the Promotion of Science.}
\affiliation{Osaka University}

\address{$^{\ddagger}$ Department of Information and Physical Science\\ Graduate School of Information Science and Technology, Osaka University\\
Suita Osaka 565-0871, Japan    \\
$^{\dagger}$ Department of Applied Physics\\
 Graduate School of Engineering, Osaka University\\
 Suita Osaka 565-0871, Japan}

\end{aug}

\begin{abstract}
This paper presents a stochastic differential equation model  for describing the process of fish schooling. The model equation always possesses a unique local solution, but global existence can be shown only in some particular cases. Some numerical examples show that the global existence may fail in general.
\end{abstract}


\end{frontmatter}

\section{Introduction}

We are interested in describing the process of fish schooling by the ordinary differential equations. A model written in terms of ODE is very useful. First, the rules of behavior of individual animals can be described precisely. Second, many techniques which have been developed in the theory of ODE can directly be available to analyse their solutions including asymptotic behavior and numerical computations.

We will regard the fish as particles in the space $\mathbb R^d$. The direction in which a fish proceeds is regarded as its forward direction. As for the assumptions of modeling, we will follow the idea presented by Camazine-Deneubourg-Franks-Sneyd-Theraulaz-Bonabeau \cite{CDFSTB} which is also based on empirical results Aoki \cite{Ao}, Huth-Wissel \cite{HW} and Warburton-Lazarus \cite{WL}. In the monograph \cite[Chapter 11]{CDFSTB}, they have made the following assumptions:
\begin{enumerate}
\item  The school has no leaders and each fish follows the same behavioral rules.
\item  To decide where to move, each fish uses some form of weighted average of the position and orientation of its nearest neighbors.
\item  There is a degree of uncertainty in the individual's behavior that reflects both the imperfect information-gathering ability of a fish and the imperfect execution of the fish's actions.
\end{enumerate}
We remark that similar assumptions, but deterministic ones, were also introduced by Reynolds \cite{Re}. 

As seen in Section 2, we formulate the motion of each individual by a system of deterministic and stochastic differential equations. The weight of average is taken analogously to the law of gravitation. That is, for the $i$-th fish at position $x_i$, the interacting force with the $j$-th one at $x_j\, (i \not= j)$ is given by
\begin{equation*}
  - \alpha \Big[\frac1{(\|x_i - x_j\|/r)^p}
    - \frac1{(\|x_i-x_j\|/r)^q}\Big](x_i-x_j),
\end{equation*}
where $1 < p < q < \infty$ are some fixed exponents and $r > 0$ is a critical radius. This means that if $x_i$ and $x_j$ are far enough that $\|x_i-x_j\| > r$, then the interaction is attractive; conversely, if it is opposite $\|x_i-x_j\| < r$, then the interaction is repulsive. The exponents $p,\, q$ and the radius $r$ may depend on the species of animal. The larger $p$ and $q$ are, the shorter the relative range of interactions between two individuals.

A similar weight of average is used for the orientation matching, too, i.e., 
$$-\beta \Big[ \frac1{(\|x_i-x_j\|/r)^p} +
       \frac1{(\|x_i-x_j\|/r)^q} \Big](v_i-v_j).$$
Here, $v_i$ and $v_j$ denote velocities of the $i$-th and $j$-th animals, respectively.

Several kinds of mathematical models have already been presented,  including difference or differential models. Vicsek et al. \cite{VCBO} introduced a simple difference model, assuming that each particle is driven with a constant absolute velocity and  the average direction of motion of the particles in its neighborhood together with some random perturbation.  Oboshi et al. \cite{OKMI} presented another difference model in which an individual selects one basic behavioral pattern from four  based on the distance between it and its nearest neighbor. Finally, Olfati-Saber \cite{ROS} and D\'{}Orsogna et al. \cite{OCBC} constructed a deterministic differential model using a generalized Morse and attractive/repulsive potential functions, respectively.

In this paper, after introducing the model equations, we shall prove local existence of solutions and in some particular cases global existence, too. We shall also present some numerical examples which show robustness of the behavioral rules introduced in \cite[Chapter 11]{CDFSTB} for forming a swarm against the uncertainty of individual's information processing and executing its actions.

In the forthcoming paper, we are going to construct a particle swarm optimization scheme on the basis of the behavioral rules of swarming animals which can spontaneously and successfully find their feeding stations.

The organization of the present paper is as follows.  In the next section, we show our model equations. Section 3 is devoted to  proving local existence of solutions. Section 4 gives  global existence for both deterministic and stochastic cases but the number of animal is only two. Some numerical examples that suggest global existence is not true in general are presented in Section 5.
\section{Model Equations}
We consider motion of $N$ fish. They are regarded as moving particles in the space $\mathbb R^d\, (d=1,2,3,\ldots)$. The position of the $i$-th particle is denoted by $x_i = x_i(t)\, (i=1,2,\ldots,N)$. Its velocity is denoted by $v_i = v_i(t)\, (i=1,2,\ldots,N)$. Our model is then  given by
\begin{equation} \label{eq}
\left\{ \begin{aligned}
  &dx_i = v_i dt + \sigma_i dw_i,    \\
  &dv_i = \Big\{  - \alpha \sum_{j=1,\, j\not=i}^N
    \Big[\frac1{(\|x_i - x_j\|/r)^p}
         - \frac1{(\|x_i-x_j\|/r)^q} \Big](x_i-x_j)   \\
  &{}\qquad\quad\enskip       - \beta \sum_{j=1,\, j\not=i}^N
    \Big[ \frac1{(\|x_i-x_j\|/r)^p} +
       \frac1{(\|x_i-x_j\|/r)^q} \Big](v_i-v_j) \\
  &{}\qquad\quad\enskip       
     + F_i(t,x_i,v_i) \Big\} dt.
\end{aligned} \right.
\end{equation}
The first equation is a stochastic equation on $x_i$, where $\sigma_idw_i(t)$ denotes a noise resulting from the imperfectness of information-gathering and action of the fish.  In fact, $\{w_i(t), t\geq 0\} (i=1,\dots, N)$ are  independent $d$\,{-}\,dimensional Brownian motions defined on a complete probability space 
 with filtration $(\Omega, \mathcal F, \{\mathcal F_t\}_{t\geq 0},\mathbb P)$ satisfying the usual conditions.  The second one is a deterministic equation on $v_i$, where $1 < p < q < \infty$ are fixed exponents, $r > 0$ is a fixed radius and $\alpha,\, \beta$ are  positive constants. Finally, $F_i(t,x_i,v_i)$ denotes an external force at time $t$ which is a given function defined for $(x_i,v_i)$ with values in $\mathbb R^d$. It is assumed that $F_i(t,x_i,v_i)\, (i=1,\dots, N)$ are locally Lipschitz continuous. 

In what follows, for simplicity, we shall put $\alpha_1=\alpha r^p, \beta_1=\beta r^p, \gamma=r^{q-p}.$ Then, the system \eqref{eq} is rewritten in the form
\begin{equation}\label{H2}
\left\{ \begin{aligned}
&dx_i=v_idt +\sigma_i dw_i,\\
&dv_i=\Big\{ -\alpha_1 \sum_{j=1, j\ne i}^N\Big[\frac{1}{||x_i-x_j||^p}-\frac{\gamma}{||x_i-x_j||^q}\Big](x_i-x_j)\\
&{}\qquad\quad\enskip -\beta_1 \sum_{j=1, j\ne i}^N\Big[\frac{1}{||x_i-x_j||^p}+\frac{\gamma}{||x_i-x_j||^q}\Big](v_i-v_j) \\
  &{}\qquad\quad\enskip       + F_i(t,x_i,v_i)\Big\}dt,
\end{aligned}\right.
\end{equation}
for $i=1,\ldots, N.$
\bigskip

\section{Local Solution}
We set the phase space
\begin{align*}
\mathbb R(N)=
\{&(x_1,\ldots, x_N, v_1,\ldots,v_N) \in \mathbb R^{Nd}\times \mathbb R^{Nd} \,{|}\,x_i\ne x_j \\
& (1\leq i, j \leq N, i\ne j) \}. 
\end{align*}
Since all the functions in the right hand side of \eqref{H2}  are locally Lipschitz continuous in $\mathbb R(N)$, the existence and uniqueness of local solutions to  \eqref{H2} starting from points belonging to this phase space are obvious in both deterministic and stochastic cases, see for instance \cite{A,F}. Thus, we have
\begin{theorem} \label{thm0}
For any initial value 
$$
  (x_1(0),\ldots, x_N(0), v_1(0),\ldots,v_N(0)) \in \mathbb R(N),
$$
\eqref{H2} has a unique local solution defined on an interval $[0, \tau)$ with values in $\mathbb R(N)$, where $\tau \le \infty$ and if $\tau < \infty$ it is an explosion time.
\end{theorem}
\section{Global solution in some particular cases}

In this section, we shall consider the case where $N=2$ and prove global existence for  \eqref{H2}. First, the deterministic case (i.e., $\sigma_1=\sigma_2=0$) is treated with null external forces $F_1=F_2 \equiv 0$. Second, the stochastic case (i.e., $\sigma_1+\sigma_2>0$) is treated but under the restriction that $d$ and $q$ satisfy the relations  $d> \max\{q-4, 2\}$ and $ q>2$ (therefore, in particular, $d>2$).
\subsection{ Deterministic case ($\sigma_1=\sigma_2=0) $}
The system  \eqref{H2} has the form
\begin{equation}\label{H3}
\begin{cases}
\begin{aligned}
  \frac{dx_i}{dt}=&v_i,\\
  \frac{dv_i}{dt}=&-\frac{\alpha_1 (x_i-x_j)}{||x_i-x_j||^p}
   + \frac{\alpha_1 \gamma (x_i-x_j)}{||x_i-x_j||^q}
   - \frac{\beta_1(v_i-v_j)}{||x_i-x_j||^p}
   - \frac{\beta_1 \gamma(v_i-v_j)}{||x_i-x_j||^q},
\end{aligned}
\end{cases}
\end{equation}
where $i,j=1, 2, i\ne j.$
\begin{theorem} \label{Th1} 
Let $1<p<q<\infty$ and  $q>2$. Then, for any initial value $(x^0, v^0)\in \mathbb R(2),$  \eqref{H3} has a unique global solution $(x(t), v(t))$ with values in $\mathbb R(2)$.
\end{theorem}

\begin{proof}
As stated in Theorem \ref{thm0},  there is a unique solution $(x_1(t),x_2(t),$ $v_1(t),v_2(t))$ to  \eqref{H3} defined on an interval $[0,\tau_1),$ where $\tau_1$  denotes the explosion time. On $[0,\tau_1)$,  \eqref{H3} is equivalent to
\begin{equation*}
\begin{cases}
\begin{aligned}
  \frac{d(x_1+x_2)}{dt} &= v_1+v_2,\\
  \frac{d(v_1+v_2)}{dt} &= 0,\\
  \frac{d(x_1-x_2)}{dt} &= v_1-v_2,\\
  \frac{d(v_1-v_2)}{dt} &= -2\left[\frac{\alpha_1}{||x_1-x_2||^p}
    - \frac{\alpha_1\gamma}{||x_1-x_2||^q}\right](x_1-x_2)    \\
  &\enskip\enskip -2\left[\frac{\beta_1}{||x_1-x_2||^p}+\frac{\beta_1 \gamma}
      {||x_1-x_2||^q}\right](v_1-v_2).
\end{aligned}
\end{cases}
\end{equation*}
Thus,
\begin{equation}\label{H5}
\begin{cases}
\begin{aligned}
  x_1(t)+x_2(t) &= [v_1(0)+v_2(0)]t+x_1(0)+x_2(0), \\
  v_1(t)+v_2(t) &= v_1(0)+v_2(0),  \\
  \frac{d(x_1-x_2)}{dt} &= v_1-v_2, \\
  \frac{d(v_1-v_2)}{dt} &= -2\left[\frac{\alpha_1}{||x_1-x_2||^p}
     - \frac{\alpha_1\gamma}{||x_1-x_2||^q} \right](x_1-x_2)  \\
  &\enskip\enskip -2\left[\frac{\beta_1}{||x_1-x_2||^p}
     + \frac{\beta_1 \gamma}{||x_1-x_2||^q}\right](v_1-v_2).
\end{aligned}
\end{cases}
\end{equation}
So we put $\xi=x_1-x_2$ and $\eta=v_1-v_2$. In order to prove that $\tau_1=\infty$, it suffices to show that the solution starting in $\mathbb R^d_*=\{\xi \in \mathbb R^d\,{:}\,\xi\ne 0\}$ of the following system 
\begin{equation}\label{H6}
\begin{cases}
\begin{aligned}
  \frac{d\xi}{dt} &= \eta,\\
  \frac{d\eta}{dt} &= -2\left(\frac{\alpha_1}{||\xi||^p}
      -\frac{\alpha_1\gamma}{||\xi||^q}\right)\xi
      -2\left(\frac{\beta_1}{||\xi||^p}
      +\frac{\beta_1\gamma}{||\xi||^q}\right)\eta
\end{aligned}
\end{cases}
\end{equation}
is global. Obviously, $\tau_1$  is the explosion time of \eqref{H6}, too. Suppose that $\tau_1<\infty$. On $[0,\tau_1)$, we put $X=\frac{1}{||\xi||}, Y=||\eta||^2, Z= \langle\xi, \eta\rangle$. Then, it is easy to verify that $(X,Y,Z)$ satisfies $X(t)>0, Y(t)\geq X^2(t)Z^2(t)$ and also satisfies the following equations
\begin{equation}\label{H7}
\begin{cases}
\begin{aligned}
\frac{dX}{dt}&=-X^3Z,\\
\frac{dY}{dt}&=-4\alpha_1 X^pZ+4\alpha_1\gamma X^qZ-4 (\beta_1 X^p+\beta_1 \gamma X^q)Y,\\
\frac{dZ}{dt}&=Y-2\alpha_1 X^{p-2}+2\alpha_1\gamma X^{q-2}-2 (\beta_1 X^p+\beta_1 \gamma X^q) Z.
\end{aligned}
\end{cases}
\end{equation}
Furthermore, 
\begin{equation}\label{H7.8}
\limsup_{t\to \tau_1} [X(t)+Y(t)+|Z(t)|+X^{-1}(t)]=\infty.
\end{equation}
By introducing a function 
$$H=X^{q-4}+X^{q-2}+ Y^2+M Z^2+Y+X^{-4}+M$$
 with a sufficiently large $M> 0$, we observe that 
\begin{equation*}
\begin{aligned}
\frac{dH}{dt}=&-(q-4) X^{q-2}Z-(q-2) X^qZ\\
&+2 Y[-4\alpha_1 X^pZ+4\alpha_1\gamma X^qZ-4 (\beta_1 X^p+\beta_1 \gamma X^q)Y]\\
&+2M Z[Y-2\alpha_1 X^{p-2}+2\alpha_1\gamma X^{q-2}-2 (\beta_1 X^p+\beta_1 \gamma X^q) Z] \\
&-4\alpha_1 X^pZ+4\alpha_1\gamma X^qZ-4 (\beta_1 X^p+\beta_1 \gamma X^q)Y+4X^{-2}Z\\
=&-(q-2) X^qZ-8 \alpha_1 X^pYZ+8 \alpha_1 \gamma X^qYZ +2MYZ - 4M\alpha_1  X^{p-2} Z \\
&+(4M\alpha_1\gamma-q+4) X^{q-2} Z -4 \alpha_1 X^p Z +4\alpha_1 \gamma X^qZ\\
&-8(\beta_1 X^p+\beta_1 \gamma X^q)Y^2-4M (\beta_1 X^p+\beta_1 \gamma X^q)Z^2\\
&-4 (\beta_1 X^p+\beta_1 \gamma X^q)Y+4X^{-2}Z.
\end{aligned}
\end{equation*}
It is easily seen that, for a sufficient small $\epsilon >0$, it holds true that
\begin{align*}
&\epsilon X^qY +\epsilon^{-3} X^{q-2} \geq \epsilon X^{q+2} Z^2+\epsilon^{-3} X^{q-2} \geq 2 \epsilon^{-1}X^q |Z|, \\
&\epsilon X^pY^2 + M\beta_1 \gamma X^p Z^2 \geq 2 \sqrt{\epsilon M\beta_1 \gamma} X^p Y|Z|,\\ 
&\epsilon X^qY^2 + M\beta_1 \gamma X^q Z^2 \geq 2 \sqrt{\epsilon M\beta_1 \gamma} X^q Y|Z|,\\
&X^{q}|Z| + Z^2+\frac{1}{4\epsilon^2} \geq (X^{q}+\epsilon^{-1})|Z| \geq \max \{ X^q |Z|, X^{p} |Z|, X^{q-2} |Z|\},\\
&  X^{-4} + Z^2 \geq 2 X^{-2} |Z|, Z^2+M^2\geq 2M |Z|,\\
&(X^{q}+\epsilon^{-1})|Z| \geq X^{p-2}|Z| \quad \quad \quad\quad(\text {if } p\geq 2),\\
&(M^2X^{-2}+1) |Z| \geq M X^{p-2} |Z|\quad \quad (\text {if } p< 2).
\end{align*}
In addition, it is clear that  
$MY^2+MZ^2\geq 2MYZ.$
Then it follows that there exists $M_1>0$ such that  for  $X>0, Y\geq X^2Z^2, Z\in \mathbb R,$ $\frac{dH}{dt}$ is estimated by $ \frac{dH}{dt}\leq M_1 H$ on $[0,\tau_1)$. Therefore, by the comparison theorem, we obtain 
$$H(t)=X^{q-4}(t)+X^{q-2}(t)+ Y^2(t)+M Z^2(t)+Y(t)+X^{-4}(t)+M\leq H(0) e^{M_1\tau_1} $$ for all $t\in [0,\tau_1)$. Thus, due to \eqref{H7.8}, $\tau_1=\infty.$ Therefore, the solution of  \eqref{H3} must be global.
\end{proof}

\subsection{Stochastic case ($\sigma_1+\sigma_2>0) $}

In this subsection, we consider the stochastic case. The system  \eqref{H2} becomes 
\begin{equation}\label{H4}
\left\{ \begin{aligned}
 &dx_i = v_i dt + \sigma_i dw_i(t),\\
 &dv_i = \Big\{-\left[\frac{\alpha_1}{||x_i-x_j||^p}
   - \frac{\alpha_1 \gamma}{||x_i-x_j||^q}\right](x_i-x_j)   \\
 &{}\qquad\quad
   -\left[\frac{\beta_1}{||x_i-x_j||^p}
   +\frac{\beta_1 \gamma}{||x_i-x_j||^q}\right](v_i-v_j)\Big\}dt,
\end{aligned}
\right.
\end{equation}
where $i,j=1, 2, i\ne j.$
For \eqref{H4} the situation is not similar to that of the deterministic case. Precisely, if  $d>\max\{q-4, 2\}$ and  $q>2$ then the global existence is shown, while if $d=1$ or $2$ then some solution may explode at a finite time. 
\begin{theorem}\label{Th3}
Let $d>\max\{q-4, 2\}$ and  $q>2$. Then, for any initial value $(x^0, v^0)\in \mathbb R(2),$  \eqref{H4} has a unique global solution in $\mathbb R(2)$.
\end{theorem}

\begin{proof}
From Theorem \ref{thm0}, there exists a local solution of \eqref{H4} defined on $[0, \tau_1^*),$ where $\tau_1^*$ is an explosion time. In that interval we have
\begin{equation*}
\begin{cases}
\begin{aligned}
  v_1(t)+v_2(t) &= v_1(0)+v_2(0),\\
  x_1(t)+x_2(t) &= [v_1(0)+v_2(0)]t+\sigma_1 w_1(t)+\sigma_2 w_2(t)
          +x_1(0)+x_2(0), \\
  d(x_1-x_2) &= (v_1-v_2)dt + \sigma_1 dw_1(t)-\sigma_2 dw_2(t), \\
  d(v_1-v_2) &= \Big\{ -2\left[\frac{\alpha_1}{||x_1-x_2||^p}
          -\frac{\alpha_1\gamma}{||x_1-x_2||^q} \right](x_1-x_2)  \\
  &{}\quad -2\left[\frac{\beta_1}{||x_1-x_2||^p}
          +\frac{\beta_1 \gamma}{||x_1-x_2||^q}\right](v_1-v_2)\Big\} dt.
\end{aligned}
\end{cases}
\end{equation*}
Then $\tau_1^*$ becomes an explosion time of the following system
\begin{equation}\label{H8}
\begin{cases}
\begin{aligned}
  d\zeta &= \psi dt+ \sigma dw(t), \\
  d\psi &= \left[-2\left(\frac{\alpha_1}{||\zeta||^p}
    -\frac{\alpha_1\gamma}{||\zeta||^q}\right)\zeta
  -2\left(\frac{\beta_1}{||\zeta||^p}
  +\frac{\beta_1\gamma}{||\zeta||^q}\right)\psi \right]dt,
\end{aligned}
\end{cases}
\end{equation}
too, where $\zeta=x_1-x_2, \psi=v_1-v_2,  \sigma =\sqrt{\sigma_1^2+\sigma_2^2}$ and  $ w(t)=\frac{1}{\sigma}[\sigma_1w_1(t)-\sigma_2 w_2(t)]$ is also a $d$\,{-}\,dimensional Brownian motion in $(\Omega, \mathcal F, \{\mathcal F_t\}_{t\geq 0},\mathbb P)$. By putting
\begin{equation} \label{H9}
X=\frac{1}{||\zeta||}, \quad Y=||\psi ||^2, \quad Z = \langle\zeta,\psi\rangle
\end{equation}
and using the It\^o formula, it is easily obtained that on $[0,\tau^*_1), (X,Y,Z)$ with $X(t)>0, Y(t)\geq 0$ satisfies the equations:
\begin{equation}\label{H10}
\begin{cases}
\begin{aligned}
dX=&\big[-X^3Z-\frac{d-3}{2} \sigma^2 X^3\big]dt -\sigma X^3\langle\zeta, dw\rangle,\\
dY=&[-4\alpha_1 X^pZ+4\alpha_1\gamma X^qZ-4 (\beta_1 X^p+\beta_1 \gamma X^q)Y]dt,\\
dZ=&[Y-2\alpha_1 X^{p-2}+2\alpha_1 \gamma X^{q-2}-2(\beta_1 X^p+\beta_1 \gamma X^q)Z]dt +\sigma  \langle\psi , dw\rangle.
\end{aligned}
\end{cases}
\end{equation}
Let us define a sequence of stopping times by putting, for each integer $k\geq k_0,$ 
$$\tau_k=\inf\left \{t\geq 0\,{:}\,X(t)\notin (\frac{1}{k},k) \text{ or } Y(t)\notin [0,k)\right\},$$
where $k_0>0$ is a sufficiently large number such that $(X(0), Y(0))\in (\frac{1}{k_0},k_0) \times [0,k_0).$ We here use  convention that the infimum of the empty set is $\infty$. Since $\tau_k$ is nondecreasing as $k \rightarrow \infty$, there exists a limit $\tau_{\infty}=\lim_{k \rightarrow \infty} \tau_k$. It is clear that $\tau_{\infty}\leq \tau_1^*$ a.s. We can in fact show that $\tau_{\infty}=\infty$ a.s. Suppose the contrary, then there would exist $T>0$ and $\varepsilon \in (0,1)$ such that $\mathbb P\{\tau_{\infty}\leq T\}>\varepsilon.$ By denoting $\Omega_k=\{\tau_k \leq T\},$ there exists $k_1\geq k_0$ such that 
\begin{equation} \label{H11}
\mathbb P(\Omega_k)\geq \varepsilon \hspace{2cm} \text{ for all } k\geq k_1.
\end{equation}
Consider the following function in $C^2 (\mathbb R^2_+\times \mathbb R,\mathbb R_+):$
\begin{equation}\label {H12}
\begin{aligned} 
V(X,Y,Z)=& X^{\theta}+ X^{-4}+ Y^2+M Z^2+ M,
\end{aligned}
\end{equation}
where $M>0$ is a sufficiently large number  and $\theta$ is a fixed exponent such that $\max\{q-6, 0\}<\theta< \min  \{d-2, q-2\}.$ 
If  $(X(t), Y(t), Z(t)) \in \mathbb R^2_+\times \mathbb R,$ by using the It\^o formula, we get
\begin{equation} \label{H13}
\begin{aligned}
dV(X(t), Y(t), Z(t))=&f(X(t), Y(t), Z(t))dt\\
&+\langle g(X(t), Y(t), Z(t), \zeta(t), \psi(t)) ,dw(t)\rangle.
\end{aligned}
\end{equation}
Here,
\begin{equation} \label{H14}
\begin{aligned}
&f(X,Y,Z)=\\
&\theta X^{\theta-1} [-X^3Z-\frac{d-3}{2}\sigma^2 X^3]+\frac{1}{2} \theta (\theta-1)\sigma^2 X^{\theta+2}+4X^{-2} Z\\
&+2(d+2)\sigma^2 X^{-2}+ 2 Y[-4\alpha_1 X^pZ+4\alpha_1\gamma X^qZ-4 (\beta_1 X^p+\beta_1 \gamma X^q)Y]\\
&+2M Z[Y-2\alpha_1 X^{p-2}+2\alpha_1\gamma X^{q-2}-2 (\beta_1 X^p+\beta_1 \gamma X^q) Z]+\sigma^2 M Y \\
=&-\theta X^{\theta+2} Z-4\alpha_1 MX^{p-2}Z+2\alpha_1 \gamma M X^{q-2} Z-8\alpha_1 X^pYZ+8\alpha_1 \gamma X^q YZ\\
&+2M YZ+M\sigma^2 Y +4X^{-2} Z+2(d+2)\sigma^2 X^{-2}\\
&-\frac{(d-2 -\theta ) \theta \sigma^2}{2}  X^{\theta+2}-8(\beta_1 X^p+\beta_1\gamma X^q) Y^2 -4M (\beta_1 X^p+ \beta_1 \gamma X^q) Z^2.
\end{aligned}
\end{equation}
And $g$ is  a suitable function. As for the deterministic case, it holds true that
\begin{align*}
&M\beta_1\gamma X^{q} Z^2 +\frac{\epsilon^{2}}{M\beta_1\gamma} X^{\theta+2} \geq  2 \epsilon X^{\frac{q+\theta+2}{2}} |Z|, \\
&\epsilon X^pY^2 + M\beta_1 \gamma X^p Z^2 \geq 2 \sqrt{\epsilon M\beta_1 \gamma} X^p Y|Z|,\\ 
&\epsilon X^qY^2 + M\beta_1 \gamma X^q Z^2 \geq 2 \sqrt{\epsilon M\beta_1 \gamma} X^q Y|Z|,\\
&Y^2+Z^2\geq 2Y|Z|, Z^2+1\geq 2|Z|, Y^2+1\geq 2Y, X^{-4}+M^2\geq 2M X^{-2},
\end{align*}
with a sufficiently small $\epsilon>0$.  
When $p\geq 2$, since $\frac{q+\theta+2}{2}> \max \{q-2, \theta+2\}>p-2\geq 0$,  we have 
$$2\epsilon X^{\frac{q+\theta+2}{2}}|Z|+ M_1|Z| \geq \max\{ M^2 X^{\theta+2} |Z|, M^2 X^{p-2} |Z|, M^2 X^{q-2} |Z|\}$$
with a sufficiently large $M_1>0.$ 
Meanwhile, when $1<p<2$, we have 
$$MX^{-4}+2M Z^2 +M^3 \geq 2M X^{-2}|Z|+2M^2 |Z| \geq 2M X^{p-2} |Z|.$$
Thus, whatever $p$ is, there exists $M_2>0$ such that 
$$f(X,Y,Z)\leq M_2 V(X,Y,Z)  \text{ for every }X>0, Y\geq X^2Z^2, Z\in \mathbb R.$$
Since for every $t\geq 0$ it holds true from \eqref{H9} and \eqref{H13} that
$$
  (X(t\wedge \tau_k),Y(t\wedge \tau_k), Z(t\wedge \tau_k))
  \in \mathbb R^2_+ \times \mathbb R, \quad
  Y(t\wedge \tau_k) \geq X^2(t\wedge \tau_k)Z^2(t\wedge \tau_k),
$$
we have
\begin{equation*}
\begin{aligned}
\int_0^{s\wedge \tau_k}dV(X(t),Y(t), Z(t))\leq&\int_0^{s\wedge \tau_k} M_2 V(X(t),Y(t), Z(t))dt\\
&+ \int_0^{s\wedge \tau_k} \langle g(X(t),Y(t), Z(t)) ,dw_1(t)\rangle.
\end{aligned}
\end{equation*}
Taking the expectation on both side of this inequality gives
$$d\mathbb E V(X(s\wedge \tau_k),Y(s\wedge \tau_k), Z(s\wedge \tau_k))  \leq M_2 \mathbb E V(X(s\wedge \tau_k),Y(s\wedge \tau_k), Z(s\wedge \tau_k)) ds, $$
from which it follows that for every $s\geq 0$,
$$\mathbb E V(X(s\wedge \tau_k),Y(s\wedge \tau_k), Z(s\wedge \tau_k))\leq V(X(0), Y(0),Z(0))e^{M_2s}.$$
In particular,
\begin{equation} \label{H15}
\mathbb E V(X(T\wedge \tau_k),Y(T\wedge \tau_k), Z(T\wedge \tau_k))\leq V(X(0), Y(0),Z(0))e^{M_2T}.
\end{equation}
On the other hand, for every $\omega\in \Omega_k,$  $X(\tau_k)(\omega)\in \{k, \frac{1}{k}\}$  or $Y(\tau_k)(\omega)=k$. Then,
$$V(X(T\wedge \tau_k),Y(T\wedge \tau_k), Z(T\wedge \tau_k)) \geq a_k,$$
where
$a_k= \min\left\{k^{\theta}, k^4, k^2 \right\}.$ Combining this with  \eqref{H11}, we obtain that 
\begin{align*}
\mathbb E V(X(T\wedge \tau_k),Y(T\wedge \tau_k), Z(T\wedge \tau_k))&\geq \mathbb E[1_{\Omega_k}V(x_{T\wedge \tau_k},y_{T\wedge \tau_k})]\geq \varepsilon a_k.
\end{align*}
Therefore, due to  \eqref{H15}, 
$V(X(0), Y(0),Z(0))e^{M_2T}\geq \varepsilon a_k.$
Letting $k \rightarrow \infty$ we arrive at a contradiction  $\infty > V(X(0), Y(0),Z(0))e^{M_2T} \geq \infty.$ Thus 
$\tau_{\infty}=\infty$ a.s. and consequently, $\tau_1^*=\infty$ a.s. The proof is now complete.
\end{proof}

\section{Numerical examples}

In this section, we present some numerical results. First, we give examples which shows robustness of fish schooling; second, examples which suggest possibility of collision.

\subsection{Robustness}
Let us first observe examples that show that, if $\sigma_i$ are all sufficiently small, then the schooling is strongly robust.

Set $\alpha=1, \beta=0.5, r=1, p=3, q=4, \sigma_i = 0.015$, and $F_i(t,x_i,v_i)=-5v_i$. We consider $100 (=N)$ particles in the $d$-dimensional space, where $d=2,3$. An initial value $\boldsymbol x(0)$ is generated randomly in $[0,10]^{100d}$ and $\boldsymbol v(0)\equiv 0$. Figure 1 illustrates positions of particles and their velocity vectors at $t=0, 5, 10, 15$ in $\mathbb R^2$. Figure 2 does the same  at $t=0, 10, 20, 30$ in $\mathbb R^3$.

\newpage

 \begin{figure}[h]
\includegraphics[scale=0.9]{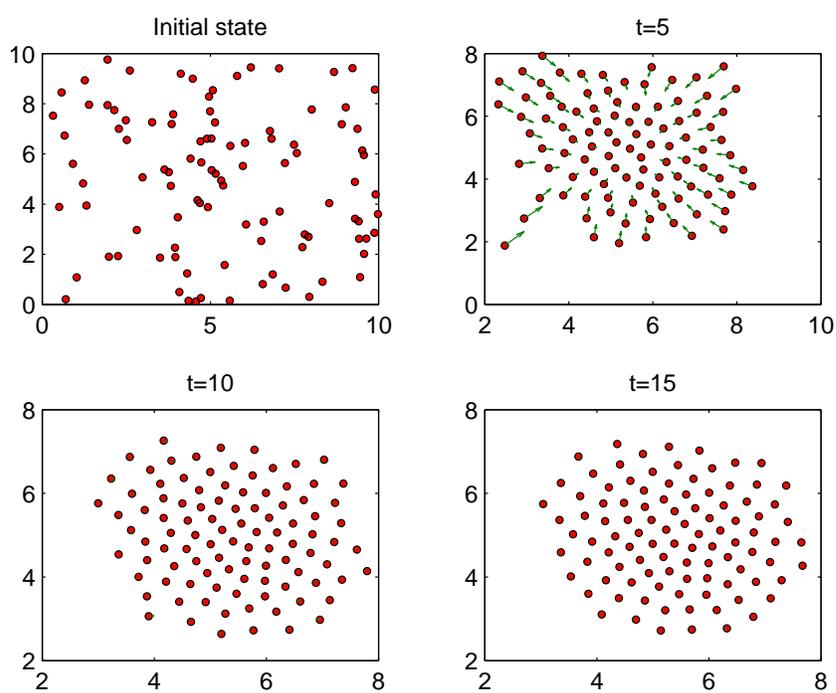}
 \caption{ Schooling in $2$-dim. space}
\end{figure}

\newpage

 \begin{figure}[h]
\includegraphics[scale=0.53]{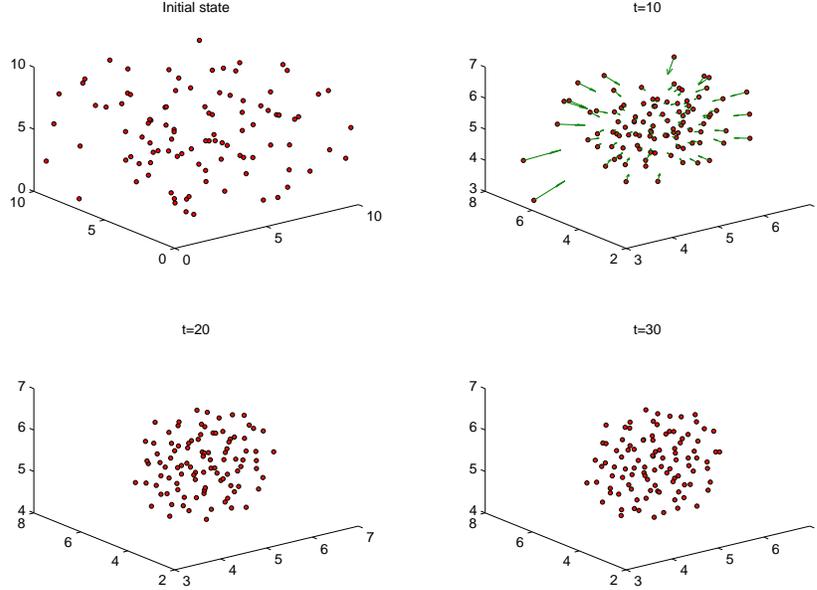}
 \caption{Schooling in $3$-dim. space}
\end{figure}
\subsection {Collision}
Let us next observe examples suggesting collision of two particles in the $d$-dimensional space, where $d=1,2$, with sufficiently small initial distance when $\sigma_i$ are not so small. 

For the case $d=1$, we set $\alpha=5, \beta=1, r=0.5, p=3, q=4, \sigma_i = \sigma,$  and  $F_i(t,x_i,v_i) = -v_i$. An initial value $\boldsymbol x(0)$  is generated randomly in $[0,1]^2$ and $\boldsymbol v(0)\equiv 0$. Figure 3 illustrates trajectories of two particles when $\sigma= 0, 0.15, 5.$  If $\sigma$ is small (i.e., $\sigma= 0, 0.15$),  collision does not take place. Meanwhile if  $\sigma$ is large ($\sigma= 5$), we observe that collision  takes place.

For the case $d=2$, we set $\alpha=7, \beta=19, r=1, p=3, q=4, \sigma_i = 9$, and $F_i(t,x_i,v_i) = -5v_i$. An initial value $\boldsymbol x(0)$ is generated randomly in $[0,5]^4$ and $\boldsymbol v(0)\equiv 0$. Figure 4 illustrates behavior of the distance of the two particles $x_1$ and $x_2$.

\newpage

\begin{figure}[h]
\includegraphics[scale=0.6]{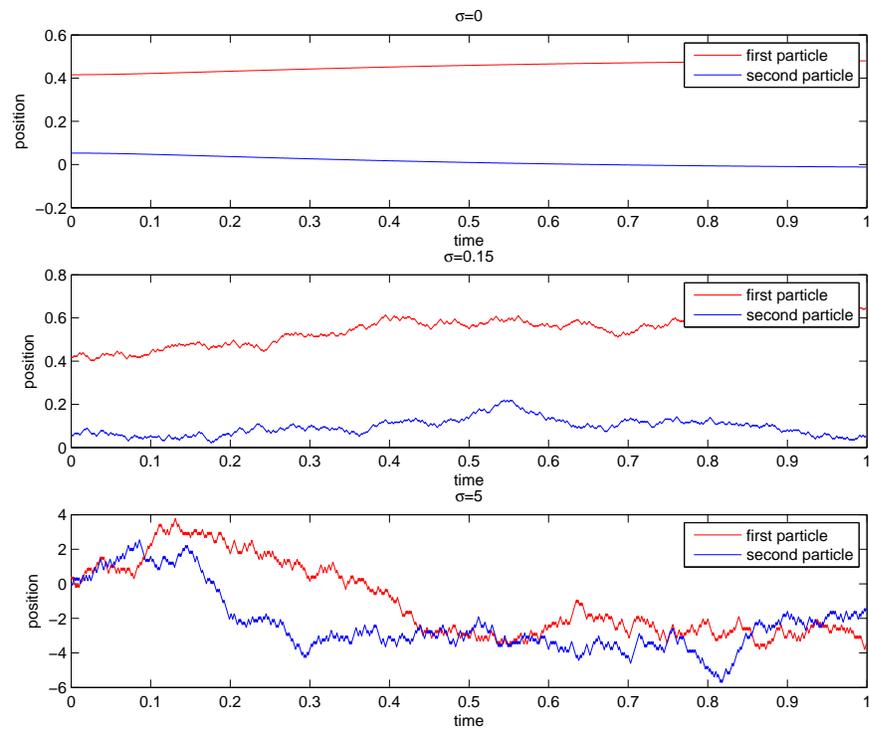}
 \caption{Collision in 1-dim. space}
\end{figure}

\newpage

\begin{figure}[h]
\includegraphics[scale=0.8]{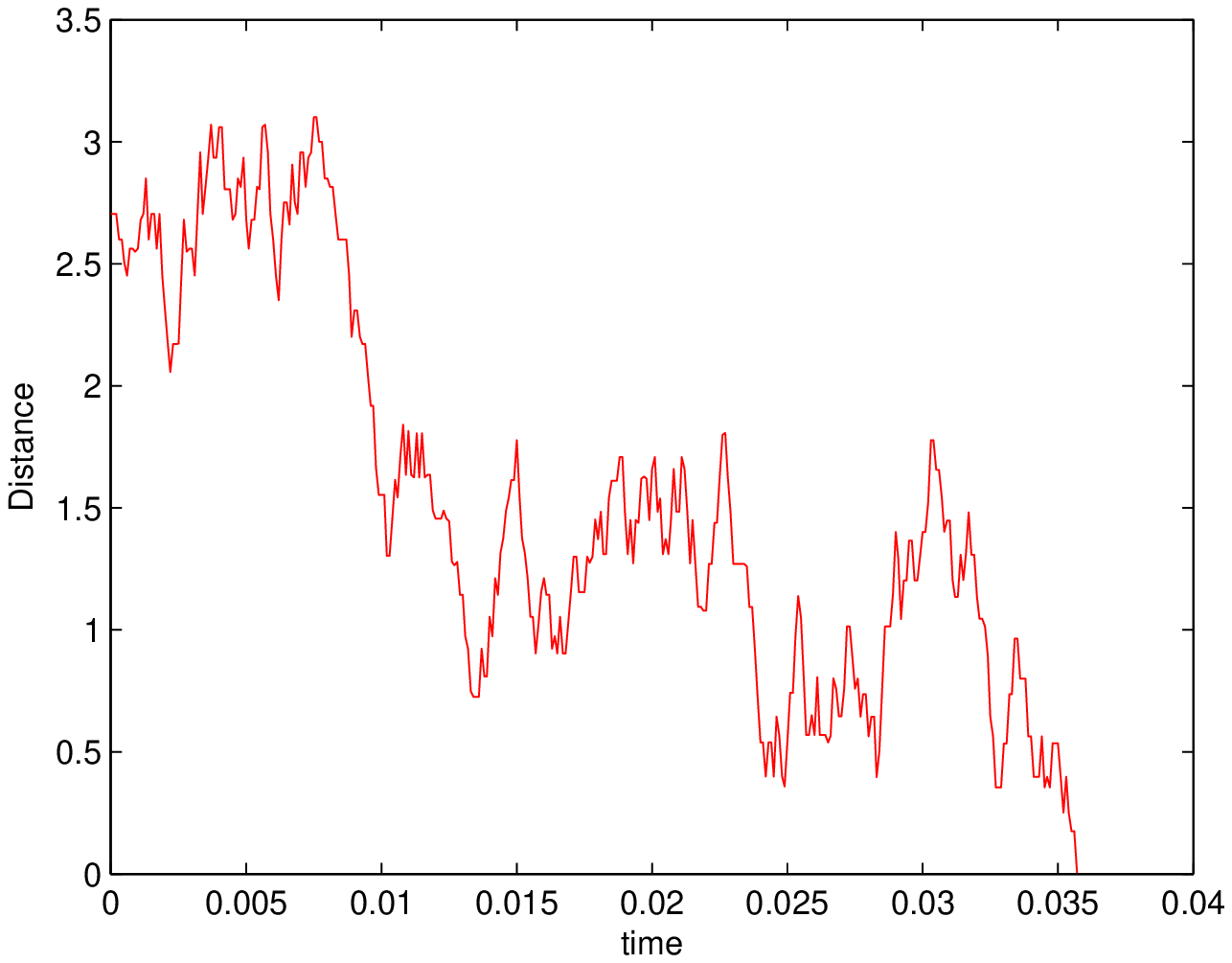}
 \caption{Collision in 2-dim. space}
\end{figure}

\end{document}